\documentclass{article}
\usepackage{amsmath}
\usepackage{amstext}
\usepackage{mathtools}
\usepackage{amssymb}
\usepackage{textcomp}
\usepackage{listings}
\usepackage{graphicx}
\usepackage{physics}
\usepackage{relsize}
\usepackage{exscale}
\usepackage{amsthm,amsmath,amssymb}
\usepackage{amsmath,mathdots}
\usepackage{bbm}
\usepackage{mathrsfs}
\usepackage{tikz-cd}
\usepackage{centernot}
\newcommand{\s}[1]{\textup{\textsf{#1}}}

\def\0{\emptyset}

\renewcommand{\L}{\langle L,<\rangle}
\newcommand{\Rr}{\langle \R,<\rangle}
\newcommand{\Rrplus}{\langle\R_{\ge 0},<\rangle}
\newcommand{\Rrminus}{\langle\R_{< 0},<\rangle}
\newcommand{\continuum}{2^{\aleph_0}}

\newcommand{\Lzero}{\langle L_0,<_0\rangle}
\newcommand{\Lone}{\langle L_1,<_1\rangle}

\newcommand{\Q}{\mathbb{Q}}
\newcommand{\R}{\mathbb{R}}

\newcommand{\im}{\mathbin{\hbox{\tt\char'42}}}

\newcommand{\dom}{\text{dom}}

\newcommand{\twoalphalex}{\langle {}^\alpha 2,<_{\text{lex}}\rangle}

\newcommand{\twoomegalex}{\langle {}^\omega 2,<_{\text{lex}}\rangle}
\newcommand{\twoomega}{{}^{\omega}2}
\newcommand{\lessomega}{{}^{<\omega}2}
\newcommand{\lex}{<_\text{lex}}

\newcounter{thmcount}
\newtheorem{thm}[thmcount]{Theorem}
\newtheorem{prop}[thmcount]{Proposition}
\newtheorem{defn}[thmcount]{Definition}
\newtheorem{lemma}[thmcount]{Lemma}
\newtheorem{cor}[thmcount]{Corollary}
\newtheorem{obs}[thmcount]{Observation}
\newtheorem{q}{Question}
\theoremstyle{definition}

\newtheorem{claim}[thmcount]{Claim}
\newcommand*{\claimproofname}{Proof of claim}
\newenvironment{claimproof}[1][\claimproofname]{\begin{proof}[#1]}{\end{proof}}

\counterwithin{thmcount2}{section}

\newtheorem{thm*}{Theorem}

\usepackage{lipsum}
\newcommand\blfootnote[1]{%
  \begingroup
  \renewcommand\thefootnote{}\footnote{#1}%
  \addtocounter{footnote}{-1}%
  \endgroup
}

\usepackage{cite}
\usepackage[shortlabels]{enumitem}
\usepackage[new]{old-arrows}
\usepackage{slashed}
\begin{document}
	\title{Infinite-Exponent Partition Relations\\on the Real Line}
	\author{Lyra A.\ Gardiner\\\normalsize{University of Cambridge}\\\normalsize{\texttt{lag44@cam.ac.uk}}}
	\date{}
	\maketitle
    
	The theory of infinite-exponent partition relations on ordinals is well developed. The Axiom of Choice implies that no such relation can hold; as such, we work in \s{ZF} without the Axiom of Choice throughout this paper.
    
    We extend the theory of infinite-exponent partition relations to arbitrary linear order types, with a particular focus on the real number line. We give a complete classification of all consistent partition relations on the real line with countably infinite exponents (Theorem \ref{mainthm}) and a characterisation of the statement ``no uncountable-exponent partition relations hold on the real line" (Theorem \ref{solovay:theorem}).
	\section{Introduction}\label{introduction:section}
    \subsection{Background and main results}\label{background:subsection}
    Partition relations were introduced by Erdős and Rado in \cite{erdosrado}, where they proved that, assuming the Axiom of Choice, the exponent of any such relation cannot be an infinite cardinal; in \s{ZF} without \s{AC}, however, so-called \textit{infinite-exponent partition relations} (IEPRs) can consistently hold. In this paper we work in \s{ZF} unless otherwise stated and extend the study of IEPRs to the setting of linear orders. Our main results are the following two theorems, whose notation will be introduced below:
	\begin{thm}
		\label{mainthm}
		Statements of the form $\Rr \rightarrow (\tau)^\tau$ with $\tau$ a countably infinite order type can hold if and only if $\tau$ is one of the order types $\omega + k$ or $k + \omega^*$ for some $k \in \omega$, and for these $\tau$, 
		\[\Rr \rightarrow (\tau)^\tau \iff \omega \rightarrow(\omega)^\omega.\]
	\end{thm}
	\begin{thm}
		\label{solovay:theorem}
		The statement ``$\Rr \centernot \rightarrow (\tau)^\tau$ for all uncountable order types $\tau$" holds if and only if every infinite set of reals is inexact. Moreover, this is implied by the statement that every set of reals has the Perfect Set Property.
	\end{thm}
	The structure of the paper is as follows. In \S\ref{terminologyandnotation:subsection}, we introduce the notation and terminology used in these theorems; in \S\ref{choice:subsection}, we explain in detail why our base theory is \s{ZF} without the Axiom of Choice; in \S\ref{motivation:subsection} we then motivate our problem of determining when $\Rr \rightarrow (\sigma)^\tau$ for linear order types $\sigma,\tau$ by explaining how we arrived at it from a question in combinatorial set theory which was not immediately obviously related to linear orders. The equivalence $\Rr \rightarrow (\tau)^\tau \iff \omega \rightarrow (\omega)^\omega$ for $\tau$ of the form $\omega + k$ or $k + \omega^*$ with $k \in \omega$ will be proved in \S\ref{positivecountable:section}; the remainder of Theorem \ref{mainthm} (i.e.\ the result that $\Rr \centernot \rightarrow (\tau)^\tau$ for all other countably infinite $\tau$) as well as Theorem \ref{solovay:theorem} will be proved in \S\ref{negative:section}. In \S\ref{largerhomogeneous:section} we fully classify the relation $\Rr \rightarrow (\sigma)^\tau$ with $\tau$ countably infinite under the assumption of Countable Choice for sets of reals. With the notable exception of Proposition \ref{schilhan:prop} in \S3, our arguments are largely combinatorial rather than set-theoretic in nature and can be followed without knowledge of advanced set theory.
	
	\subsection{Terminology and notation}\label{terminologyandnotation:subsection}
    A \textit{(linear) order type} is an isomorphism class of linear orders; we use the letters $\sigma$, $\tau$, $\varphi,\psi$ for these.\footnote{As defined these are not sets; this is not important for our purposes, but we can treat them as sets by means of Scott's trick.} For $\L$ a linear order, say that $\L$ \textit{has order type} $\sigma$ or \emph{is ordered as} $\sigma$ if $\L$ belongs to the equivalence class $\sigma$. Whenever a property of linear orders is invariant under isomorphism, it can be treated as a property of the corresponding order type. Order types are quasi-ordered by the \textit{embeddability relation} $\le$, where $\sigma \le \tau$ if whenever $\langle L_0,<_0\rangle$ has order type $\sigma$ and $\langle L_1, <_1\rangle$ has order type $\tau$, there is an $i: L_0 \to L_1$ with, for $x,y \in L_0$, $i(x) <_1 i(y)$ iff $x <_0 y$. For $\sigma$ an order type, $\sigma^*$ denotes its reverse, i.e.\ the order type of $\langle L,>\rangle$ if $\sigma$ is the order type of $\L$. Ordinals will be treated as order types rather than as transitive sets, and we use the letters $\alpha$, $\beta$, $\gamma$, $\delta$ for them. We use the letters $\kappa$, $\mu$, $\nu$ for cardinals, which will be treated as equivalence classes of sets in bijection with each other rather than as initial ordinals, as we are working in \s{ZF} and not every set we consider is well-orderable.\footnote{See footnote 1.} The letter $\lambda$ refers to the order type of the reals, treated as the real line $\Rr$.\footnote{This makes some of our arguments slightly easier to phrase, but the reader is welcome to think of the paper as being about whichever version of the reals they prefer (e.g.\ Baire space, Cantor space); all three of $\Rr$, $\langle {}^\omega\omega,\lex\rangle$, and $\twoomegalex$ are embeddable in each other, so all of the results of the paper still hold.} We will sometimes write $\R_{\ge 0}$ for the set $[0,\infty)$ and $\R_{< 0}$ for the set $(-\infty,0)$. 
    Subsets of $\R$ will be treated as linear orders equipped with the induced suborder. A \emph{real type} is an order type of a subset of $\R$, i.e.\ an order type $\tau$ with $\tau \le \lambda$. The letter $\omega$ will usually refer to the order type of $\langle \mathbb N,<\rangle$, but as is standard in set theory, we will sometimes write $\omega$ to mean the set $\mathbb N = \{0,1,2,\dots\}$ itself -- there will never be any ambiguity as to which sense of $\omega$ is meant. We will write $\aleph_0$ to mean the cardinality of $\mathbb N$, distinct from the order type $\omega$. Addition will mean addition of linear orders in the usual sense of concatenation, with the one exception of the definition of indecomposable cardinals just before Proposition \ref{indecomposable:prop}, where ``$+$" will be used to refer to addition of cardinals. We use the set-theoretic notation for images of functions: for $f$ a function, $X \subseteq \dom f$, 
    \[f \im X \coloneqq \{f(x) : x \in X\}.\]
	
	For $X$ a set and $\kappa$ a cardinal, write $[X]^\kappa$ to mean the set of subsets of $X$ of cardinality $\kappa$. Similarly, for $\L$ a linear order, $\sigma$ an order type, write $[\L]^\sigma$ to mean the set of all subsets of $L$ which have order type $\sigma$ in the induced suborder. When the order is understood we will abbreviate this to $[L]^\sigma$.
	
	Most of our results concern \textit{partition relations}, which we express with the Erdős-Rado \emph{partition relation symbol}. We first define this for relations on cardinals: for $\mu \ge \nu \ge \kappa$ cardinals and $\chi$ a set, the partition relation
	\[\mu \rightarrow (\nu)^\kappa_\chi,\]
	read ``$\mu$ arrows $\nu$ to the $\kappa$ with $\chi$ colours", is the statement that, for any set $X$ of cardinality $\mu$ and for any function $F: [X]^\kappa \to \chi$, thought of as a \emph{colouring} of the $\kappa$-sized subsets of $X$, there is a set $H \in [X]^\nu$ which is \emph{homogeneous} (or \emph{monochromatic}) for $F$, in the sense that $\left|F \im [H]^\kappa\right| = 1$, i.e.\ $F$ sends all elements of $[H]^\kappa$ to the same ``colour". The negation of this statement is written with $\centernot \rightarrow$ in place of $\rightarrow$.
	
	We now define partition relations for linear order types, which are the main focus of this paper. In an analogous way to the above, for $\varphi$, $\sigma$, and $\tau$ order types with $\varphi \ge \sigma \ge \tau$ and $\chi$ a set, we define
	\[\varphi \rightarrow (\sigma)^\tau_\chi\]
	to be the statement that for any linear order $\L$ of order type $\varphi$ and any function $F: [L]^\tau \to \chi$ (again thought of as a colouring of the copies of $\tau$ in $\L$), there is some $H \in [L]^\sigma$ which is homogeneous for $F$ in the sense that $\left|F \im [H]^\tau\right| = 1$. For ease of reading we will often write $\L \rightarrow (\sigma)^\tau_\chi$ in place of $\varphi \rightarrow (\sigma)^\tau_\chi$, where $\L$ is a linear order of order type $\varphi$. Our main object of study in this paper is the relation $\Rr \rightarrow (\sigma)^\tau_2$, where $\tau$ is an infinite order type.

    We have in a sense overloaded this arrow notation, as it is used both for partition relations which refer only to cardinalities and for partition relations which refer to order types; nonetheless, in any given setting it will be clear which type of partition relation is meant by considering whether the terms in the notation are linear order types or not. From Lemma \ref{omega:lemma} onwards, all of the partition relations we consider will refer to order types. A point on notation: when writing partition relations it is common to omit the set of colours $\chi$ when $\chi = 2 = \{0,1\}$. Throughout this paper we make use of this abbreviation, but we remark here that for all of our results, we could just as well have used any $2 \le \chi < \omega$.

    The separation of the four terms in the partition relation notation into two sides of an arrow serves to highlight the ``monotonicity" of the partition relation. By this we mean that a positive relation remains positive if we \emph{increase} the term on the left-hand side or, usually, if we \emph{decrease} any of the terms on the right-hand side (with respect to either the usual ordering of cardinals or the embeddability relation on order types). In particular, much of our work shows $\Rr \centernot \rightarrow (\tau)^\tau$ for certain order types $\tau$, and the failure of this ``minimal" relation means that $\Rr \centernot \rightarrow (\sigma)^\tau_\chi$ for all $\sigma, \chi$. We remark, however, that decreasing the \emph{exponent} (i.e.\ $\kappa$ or $\tau$ above) does not necessarily preserve the relation (as e.g.\ it is consistent that $\Rr \centernot \rightarrow (\omega + \omega)^\omega$ but $\Rr \rightarrow (\omega + \omega)^{\omega +1}$; we will see in \S\ref{largerhomogeneous:section} that this holds in Solovay's model).
	\subsection{A note on the Axiom of Choice}\label{choice:subsection}
	As mentioned, assuming the Axiom of Choice, relations of the form
	\[\mu \rightarrow (\nu)^\kappa\] in which the exponent $\kappa$ is infinite are necessarily false; see e.g.\ \cite[p.\ 434]{erdosrado}. The same argument easily yields the failure of all relations of the form
	\[\varphi \rightarrow (\sigma)^\alpha\]
	whose exponent $\alpha$ is any infinite \emph{ordinal}. Relations with exponent an infinite non-wellfounded order type are not all necessarily false in \s{ZFC}, 
    but all such relations involve a certain triviality, in a sense which we make precise below; the argument from \cite{erdosrado} can be modified to yield a similar result for linear orders (Proposition \ref{choiceieprstrivial:prop}).
    \begin{defn}\label{exact:defn}
        A linear order $\L$ is \emph{exact} if the only order-preserving function from $\L$ to itself is the identity function; equivalently, if $L$ has no proper subsets which are order-isomorphic to $\L$ in the induced suborder.\footnote{To the best of our knowledge, this use of the word ``exact" is due to Ginsburg, introduced in \cite{ginsburgremarks} and \cite{ginsburg}.} Say $\L$ is \emph{inexact} if it is not exact.
    \end{defn}
    Clearly every finite linear order is exact and every infinite ordinal or reverse ordinal is inexact; we return to the question of the existence of infinite exact orders in \S\ref{negative:section}. Exact order types yield positive infinite-exponent partition relations in a trivial sense:
    \begin{obs}\label{exacttrivial:obs}
        Let $\tau$ be an exact order type and let $\L$ be a linear order with $[L]^\tau$ non-empty. Then
        \[\L \rightarrow (\tau)^\tau.\]
    \end{obs}
    \begin{proof}
        Let $F : [L]^\tau \to 2$ be any colouring; then any $H \in [L]^\tau$ is homogeneous for $F$, as $[H]^\tau = \{H\}$.
    \end{proof}

    In \s{ZFC}, this is in fact an equivalence:
    \begin{prop}\label{choiceieprstrivial:prop}\emph{\textsf{(ZFC)}}
        Let $\tau$ be an inexact order type and let $\L$ be a linear order with $[L]^\tau$ non-empty. Then
        \[\L \centernot \rightarrow (\tau)^\tau.\]
    \end{prop}
    We first prove the following lemma.
    \begin{lemma}\label{increasingomegaseqs:lemma}
        Let $\L$ be an inexact linear order and write $\sigma$ for its order type. Then there exists an $\omega$-sequence $ L_0  \subsetneq  L_1  \subsetneq  L_2  \subsetneq \dots$, where each $\langle L_n, <\rangle \in [L]^\sigma$.
    \end{lemma}
    \begin{proof}
        Since $\L$ is inexact, there exists some $f: L \to L$ order-preserving other than the identity. Fix such an $f$, and let $x \in L$ be a point moved by $f$, wlog $x < f(x)$. Then in particular, for any $n \in \omega$, $f^n(x) < f^{n+1}(x)$, with the convention that $f^0$ is the identity. For $n \in \omega$, write $I_n \coloneqq [f^n(x), f^{n+1}(x))$, and write $L' \coloneqq \bigcup_{n \in \omega}I_n$. Write $\sigma'$ for the order type of $\langle L',<\rangle$.
        
        Observe that $L'$ is an interval of $\L$ and so it suffices to prove the statement for the order $\langle L',<\rangle$: if $ L'_0 \subsetneq  L'_1 \subsetneq  L'_2 \subsetneq \dots$ is such that each $\langle L'_n, <\rangle \in [L']^{\sigma'}$, we obtain $ L_0 \subsetneq L_1  \subsetneq  L_2 \subsetneq \dots$ as desired by setting $L_n \coloneqq L'_n \cup (L \setminus L')$.
    
        We find such $L'_n$ in the following way: for $n \in \omega$, set
        \[L'_n \coloneqq\bigcup_{m \le n} I_m \cup   \bigcup_{m > n} f^{m - n} \im I_m.\]
        Certainly each $\langle L'_n,< \rangle \in [L']^{\sigma'}$; it remains to show that each $L'_n \subsetneq L'_{n+1}$. Observe that if $m + k = m' + k'$ and $m \le m'$, then \[f^k \im I_m \subseteq f^{k'} \im I_{m'} \subseteq I_{m + k}.\] In particular, observing that $L'_{n+1}$ can be written as
        \[L'_{n+1} = \bigcup_{m \le n} I_m \cup I_{n+1} \cup \bigcup_{m > n} f^{m - n - 1} \im I_{m+1},\]
        we see that $L'_n \subseteq L'_{n+1}$, as each $f^{m - n} \im I_{m} \subseteq f^{m - n - 1} \im I_{m+1}$. To see that this containment is strict, note that $L'_n \,\cap\, I_{n+1} = \emptyset$, while $L'_{n+1} \,\cap\, I_{n+1} = I_{n+1}$.
    \end{proof}
    \begin{proof}[Proof of Proposition \ref{choiceieprstrivial:prop}]
        We follow \cite{erdosrado}. Let $\prec$ be a well-order of $[L]^\tau$, and define $F: [L]^\tau \to 2$ by, for $A \in [L]^\tau$,
        \[F(A) = \begin{cases*}
            0 & if $A$ is $\prec$-minimal in $[A]^\tau$;\\
            1 & otherwise.
        \end{cases*}\]
        For any $A$, it is always possible to find some $A' \in [A]^\tau$ with $F(A') = 0$; simply take $A'$ to be the $\prec$-minimal element of $[A]^\tau$. Then it will certainly also be the $\prec$-minimal element of $[A']^\tau \subseteq [A]^\tau$. 
        
        We show now that in fact it is also always possible to find $A'' \in [A]^\tau$ with $F(A'') = 1$: by Lemma \ref{increasingomegaseqs:lemma}, we can find some $A_0 \subsetneq A_1 \subsetneq A_2 \subsetneq \dots$, all elements of $[A]^\tau$; then for $n < m \in \omega$, if $F(A_n) = F(A_m) = 0$, we have $A_m \prec A_n$. It follows that some $A_n$ must take colour $1$ under $F$; otherwise, $A_0 \succ A_1 \succ A_2 \succ \dots$ is a descending $\omega$-sequence in $\prec$, contradicting the assumption that $\prec$ is a well-order. It follows that no $A \in [L]^\tau$ can be homogeneous for $F$.
    \end{proof}
    We remark that this argument from \cite{erdosrado} cannot be generalised to structural infinite-exponent partition relations on arbitrary structures (in the sense of structural Ramsey theory); it fails e.g.\ for the setting of IEPRs on graphs, as it is possible to have a graph containing proper subcopies of itself but with no $\subsetneq$-chain of subcopies as in Lemma \ref{increasingomegaseqs:lemma}. It is still true, however, that the natural generalisation of Proposition \ref{choiceieprstrivial:prop} to arbitrary structures holds; see \cite{structural} for a proof.
    
    In contrast to the above results, in the absence of the Axiom of Choice it is consistent with the usual \s{ZF} axioms of set theory that non-trivial infinite-exponent partition relations can hold.  It is in light of this that our base theory in this paper is \s{ZF} without Choice.

    We make special mention here of one weak fragment of Choice which is relevant multiple times in the paper:
    \begin{defn}\label{acomegar:defn}
        The \emph{Axiom of Countable Choice for sets of reals}, \s{AC}$_\omega(\R)$, is the statement that for any family $\{X_n : n \in \omega\}$ with each $X_n \subseteq \R$, there is a choice function $\mathcal F : \omega \to \bigcup_{n \in \omega} X_n$, i.e.\ a function $\mathcal F$ with domain $\omega$ and $\mathcal F(n) \in X_n$ for each $n \in \omega$.
    \end{defn}


    Although weaker than the full Axiom of Choice, \s{AC}$_\omega(\R)$ is still independent of \s{ZF}, and we will avoid its use wherever possible throughout this paper. We remark, however, that it is consistent with IEPRs, and holds both in Solovay's model (introduced below) and in models of the Axiom of Determinacy \s{AD}, two standard settings in which IEPRs on ordinals have been studied.\footnote{Consequences of \s{AD} include the existence of (well-ordered) cardinals $\kappa$ with $\kappa \rightarrow (\kappa)^\kappa$; in particular, \s{AD} implies that $\omega_1 \rightarrow (\omega_1)^{\omega_1}$. See \cite[Theorem 28.12]{kanamori}.}
    \begin{defn}\label{solovay'smodel:defn}
        \emph{Solovay's model} of \s{ZF + DC} is a model of set theory in which every set of reals satisfies the standard regularity properties (Lebesgue Measurability, the Baire Property, and the Perfect Set Property).\footnote{The \textit{Axiom of Dependent Choice} \s{DC} is another fragment of Choice which strictly implies \s{AC}$_\omega(\R)$; see \cite[2.4.7]{jechaxiomofchoice}.} It was constructed in \cite{solovay} by forcing over a model of \s{ZFC} with an inaccessible cardinal with the Lévy collapse forcing and then reducing to a certain inner model inside that forcing extension.\footnote{In this paper we follow \cite{kanamori} in using the inner model $L(\R)$ here. See \S11 of that book for details.}
    \end{defn}
    By a result of Mathias in \cite{mathias}, the IEPR $\omega \rightarrow (\omega)^\omega$ holds in Solovay's model.\footnote{As such, this relation is consistent with \s{ZF} relative to an inaccessible cardinal. The question of whether this inaccessible is necessary is an open problem.} The relation $\omega \rightarrow (\omega)^\omega$ will be of key importance to our discussion of IEPRs on the real line.
    We will discuss Solovay's model and the Perfect Set Property in greater detail in \S\ref{negative:section} and \S\ref{largerhomogeneous:section}.

    Finally, we mention a certain kind of set which can exist in \s{ZF} without Choice and which gives rise to another kind of trivial IEPR.

    \begin{defn}\label{dedekind:defn}
        A set $D$ is a \emph{Dedekind set} if it is infinite but Dedekind-finite, i.e.\ $D$ is infinite but is not in bijection with any of its proper subsets; equivalently, $D$ is infinite but has no countably infinite subset.
    \end{defn}
    \begin{obs}\label{dedekindtrivial:obs}
        Let $\kappa$ be the cardinality of a Dedekind set, and $\mu$ any cardinal with $\kappa \le \mu$. Then
        \[\mu \rightarrow (\kappa)^\kappa.\]
    \end{obs}
    \begin{proof}
        Let $X$ be any set of cardinality $\mu$, and let $F : [X]^\kappa \to 2$ be a colouring; then any $H \in [X]^\kappa$ is homogeneous for $F$, as $[H]^\kappa = \{H\}$.
    \end{proof}
    This is analogous to Observation \ref{exacttrivial:obs}, and we remark that any linear order on a Dedekind set is necessarily exact. Countable Choice is sufficient to rule out the existence of such sets:\footnote{See e.g.\ \cite[2.4.2]{jechaxiomofchoice} for a general proof of this fact.}
    \begin{prop}\label{nodedekindwithcc:prop}
        Under \s{AC$_\omega(\R)$}, there are no Dedekind sets which are subsets of $\R$.
    \end{prop}
    \begin{proof}
        Let $\langle O_n : n \in \omega \rangle$ be an enumeration of the open intervals in $\R$ whose endpoints are both rational. If $A \subseteq \R$ is infinite, the range of a choice function for $\{A \cap O_n : A \cap O_n \neq \emptyset, n \in \omega\}$ is a countably infinite subset of $A$.
    \end{proof} 
	\subsection{Motivation}\label{motivation:subsection}
	Much work has been done on infinite-exponent partition relations on ordinals, for instance in Kleinberg's book \cite{kleinberg}. The questions in this paper arose from considering instead what relations of the form
	\[\continuum \rightarrow (\nu)^\kappa\]
	can hold with an infinite exponent $\kappa$, in situations in which the continuum is not well-orderable.  There is no explicit reference to linear orders in this question, but it turns out to be intimately related to the order types of suborders of the real line; this in turn leads to the more granular question of which relations $\Rr \rightarrow (\sigma)^\tau$ can consistently hold in \s{ZF}.

    The results in this paper can be seen as the beginning of the project of doing for infinite exponents what the remarkable paper \cite{crtolo} did for finite exponents (although our methods are distinct).
	\section{From cardinals to order types, \emph{or} \\Positive relations with countable exponents}\label{positivecountable:section}
    As a first step, we consider the minimal relation $\continuum \rightarrow (\kappa)^\kappa$. Say that a cardinal $\kappa$ is \emph{indecomposable} if there do not exist cardinals $\mu_0$ and $\mu_1$, both strictly smaller than $\kappa$, with $\kappa = \mu_0 + \mu_1$. We remark that the cardinality of any well-orderable set is indecomposable, and so in \s{ZFC} all cardinals are indecomposable; it is not a theorem of \s{ZF} that the cardinality of the continuum is indecomposable, but this is the case in Solovay's model and in models of Determinacy.\footnote{This follows from the fact that in these models, every set of reals has the PSP (see Definition \ref{perfectset:defn}).}
    
    Our first observation is that for $\kappa$ uncountable, this additional assumption on $\kappa$ is enough to get that no partition relation on $\continuum$ with exponent $\kappa$ can hold, for reasons to do with the order types of $\kappa$-sized subsets of $\R$ in the induced suborder:
	\begin{prop}\label{indecomposable:prop}
		Let $\kappa$ be an uncountable indecomposable cardinal. Then
		\[\continuum \centernot \rightarrow (\kappa)^\kappa.\]
	\end{prop}
	\begin{proof}
		We define a colouring $F : [\R]^\kappa \to 2$ with no homogeneous set as follows: for $A \in [\R]^\kappa$,
		\[F(A) = \begin{cases*}
			0 & if $\sup A \in A$ or $\inf A \in A$;\\
			1 & otherwise,
		\end{cases*}\]
		i.e.\ $A$ is given colour $0$ iff it has a minimal or a maximal element when considered as a suborder of the real line.
		
		To show that there is no homogeneous set for $F$, we must show that any $A \in [\R]^\kappa$ has a subset $B \in [A]^\kappa$ of the opposite colour. First suppose $F(A) = 0$. Without loss of generality $\inf A \in A$. We shall demonstrate a procedure for finding a subset $B$ of $A$ with the property that $\inf B \not \in B$. Define an ordinal-indexed sequence of sets $A_\alpha$ as follows:
		\begin{align*}
			A_0 &= A;\\
			A_{\alpha+1} &= A_\alpha \setminus \{\inf A_\alpha\}\text{ for each }\alpha;\\
			A_\gamma &= \bigcap_{\alpha < \gamma}A_\alpha \text{ for limit ordinals }\gamma.
		\end{align*}
		We observe that there must be some ${\delta} < \omega_1$ for which $A_{{\delta}+1} = A_{\delta}$, i.e.\ for which $A_{\delta}$ does not contain its infimum; otherwise, $\{\inf A_\alpha : \alpha < \omega_1\}$ would be a subset of $\R$ of order type $\omega_1$, but this is impossible as no such set exists. 
        It follows that at some stage ${\delta} < \omega_1$ this process terminates; $A_{\delta}$ still has cardinality $\kappa$, since at most countably many elements of $A$ have been removed by this point and $\kappa$ is uncountable and indecomposable. The analogous procedure with ``sup" replacing ``inf" clearly also produces a $\kappa$-sized subset of $A$ not containing its supremum without adding or removing an infimum. By applying one or both of these procedures to some $A \in [\R]^\kappa$ with $F(A) = 0$, we obtain $B \in [A]^\kappa$ with $F(B) = 1$.
		
		Now suppose $F(A) = 1$. Pick any $a \in A$. Since $\kappa$ is indecomposable, it is necessarily the case that at least one of $B_1 \coloneqq(-\infty, a]\cap A$, $B_2 \coloneqq [a,\infty)\cap A$ has cardinality $\kappa$. Observing that $a = \sup B_1 = \inf B_2$ and $a \in B_1$, $a \in B_2$, it follows that (at least) one of these $B_i$ is an element of $[A]^\kappa$ with $F(B_i) = 0$.
	\end{proof}
	By Proposition \ref{indecomposable:prop} we have that, at least for indecomposable uncountable cardinals $\kappa$, the relation $\continuum \rightarrow (\kappa)^\kappa$ cannot hold, for order-theoretic reasons. What, then, if instead $\kappa = \aleph_0$? This turns out to be more interesting. Once again, order types of subsets of $\R$ are relevant: the following observation immediately tells us that we need only consider partition relations for linear orders.
	\begin{prop}
		$2^{\aleph_0} \rightarrow (\aleph_0)^{\aleph_0} \iff \Rr \rightarrow (\omega)^\omega$.
	\end{prop}
    We remind the reader that the relation on the left refers to colourings of countably infinite subsets of $\R$, while the relation on the right refers to colourings of subsets of $\R$ which are ordered as $\omega$ in the induced suborder.
	\begin{proof}
		$(\!\!\!\impliedby\!\!\!)\,$: Let $F: [\R]^{\aleph_0} \to 2$ be any colouring, and let $F':[\R]^\omega \to 2$ be its restriction to $[\R]^\omega \subsetneq [\R]^{\aleph_0}$. By assumption $\Rr \rightarrow (\omega)^\omega$, so there is some $H \in [\R]^\omega$ which is homogeneous for $F'$; but then $H$ is also homogeneous for $F$, since $[H]^{\aleph_0} = [H]^\omega$.
		
		$(\!\!\implies\!\!)\,$: Let $G:[\R]^\omega \to 2$. We define a colouring $G': [\R]^{\aleph_0}\to 2$ like so:
		\[G'(A) = \begin{cases*}
			G(A) & if $A \in [\R]^\omega$;\\
			G(-A) & if $A \in [\R]^{\omega^*}$;\\
			0 & otherwise,
		\end{cases*}\]
		where here by $-A$ we mean the set $\{-a : a \in A\}$.
		
		By assumption there is some $K \in [\R]^{\aleph_0}$ which is homogeneous for $G'$. Since $K$ is countable, at least one of $[K]^\omega$ or $[K]^{\omega^*}$ is non-empty; if $K' \in [K]^\omega$ then $K'$ is homogeneous for $G$, and if $K' \in [K]^{\omega^*}$ then $-K'$ is homogeneous for $G$.
	\end{proof}
    \begin{obs}\label{nouncountablehomog:obs}
        For $\kappa > \aleph_0$,
        \[\continuum \centernot \rightarrow (\kappa)^{\aleph_0}.\]
    \end{obs}
    \begin{proof}
        Define the colouring $F: [\R]^{\aleph_0} \to 2$ by, for $A \in [\R]^{\aleph_0}$,
        \[F(A) = \begin{cases*}
            0 & if $A$ is ordered as one of $\omega$ or $\omega^*$;\\
            1 & otherwise.
        \end{cases*}\]
        Then this $F$ can only have homogeneous sets ordered as one of $\omega$ or $\omega^*$, so in particular any set homogeneous for it must be countable.
    \end{proof}
    
	Now we come to our first main result. For this result and throughout the rest of the paper we use the following notation: $\langle q_n : n \in \omega \rangle$ is a fixed enumeration of the rationals, and for $x, y \in \R$ with $x < y$, write
    \[N(x,y) \coloneqq \min\{n \in \omega: q_n \in [x,y)\}.\]
    If $x > y$, interpret $N(x,y)$ to mean $N(y,x)$. Note that we are considering intervals which are closed below and open above. As such we have the following important property of $N(x,y)$: for any real numbers $x_0 < x_1 < \dots < x_k$,
	\[N(x_0,x_k) = \min\{N(x_i,x_{i+1}) : 0 \le i \le k-1\}.\]
	\begin{lemma}\label{omega:lemma}
		$\Rr \rightarrow (\omega)^\omega \iff \omega\rightarrow(\omega)^\omega$.
	\end{lemma}
	\begin{proof}
		Note first that the reverse implication is immediate by the monotonicity of the left-hand side of the partition relation (i.e.\ we may fix a copy of $\omega$ in the real line and restrict our attention to it), so we focus on the forward implication. Suppose $\omega \centernot \rightarrow(\omega)^\omega$, i.e.\ there is some $F : [\omega]^\omega \to 2$ such that there is no homogeneous set for $F$. We use this to build a colouring $F' :[\R]^\omega \to 2$ with no homogeneous set.
		
		For $\bar{x} = \langle x_k : k \in \omega\rangle \in [\R]^\omega$, say $\bar{x}$ is \emph{$N$-increasing} if for every $i \in \omega$, $N(x_i,x_{i+1}) < N(x_{i+1},x_{i+2})$. If $\bar{x}$ is $N$-increasing, we can treat \[N(\bar{x}) \coloneqq \langle N(x_k,x_{k+1}) : k \in \omega\rangle\] as an element of $[\omega]^\omega$. We now define our colouring. For $\bar{x} \in [\R]^\omega$, set
		\[F'(\bar{x}) = \begin{cases*}
			F(N(\bar{x})) & if $\bar{x}$ is $N$-increasing;\\
			0 & otherwise.
            \end{cases*}\]
        \begin{claim}\label{nincreasingdense:claim}
        Every $\bar{x} \in [\R]^\omega$ has an infinite subsequence which is $N$-increasing.
	\end{claim}
        \begin{claimproof}
        Consider the subsequence $\bar{y}$ of $\bar{x}$ defined like so: set $y_0$ equal to that $x_k$ which minimises $N(x_k,x_{k+1})$, and for $n > 0$, set $y_n$ equal to the value of $x_k$ which minimises $N(x_k,x_{k+1})$ among those $x_k$ with $x_k > y_{n-1}$. Then for any $n \in \omega$, $N(y_n,y_{n+1}) = N(x_{k_0},x_{k_1})$ for some $k_0 < k_1$, but $N(x_{k_0},x_{k_1}) = \min\{N(x_i,x_{i+1}) : k_0 \le i < k_1\} = N(x_{k_0},x_{k_0+1})$ by choice of $k_0$. It follows that $N(y_n,y_{n+1}) < N(y_{n+1},y_{n+2})$ for all $n \in \omega$.
	\end{claimproof}
		
        \begin{claim}\label{nincreasingbij:claim} Given an $N$-increasing $\bar{x}$, for any $A \in [N(\bar{x})]^\omega$ there is some $\bar{z}_A \in [\bar{x}]^\omega$ with $N(\bar{z}_A) = A$.\end{claim}
		
	\begin{claimproof}
    	Consider the subsequence $\bar{z}_A$ of $\bar{x}$ whose members are precisely those $x_k$ such that $N(x_k,x_{k+1}) \in A$. Then if $x_{k_0},x_{k_1}$ are consecutive members of $\bar{z}_A$ for some $k_0 < k_1$, we have $N(x_{k_0},x_{k_1}) = \min\{N(x_i,x_{i+1}) : k_0 \le i < k_1\} = N(x_{k_0},x_{k_0+1})$, as $\bar{x}$ was assumed to be $N$-increasing. It follows that $N(\bar{z}_A) = A$.
	\end{claimproof}
		Now assume we have some $\bar{x} \in [\R]^\omega$ which is homogeneous for $F'$. By Claim \ref{nincreasingdense:claim}, there is some $N$-increasing $\bar{y} \in [\bar{x}]^\omega$. Any subset of a homogeneous set is necessarily also homogeneous, so $\bar{y}$ is also homogeneous for $F'$. By assumption, there is no homogeneous set in $[\omega]^\omega$ for $F$, so in particular $N(\bar{y})$ is not homogeneous for $F$; there therefore exists some $A \in [N(\bar{y})]^\omega$ with $F(A) \neq F(N(\bar{y}))$. But now by Claim \ref{nincreasingbij:claim} there is some $\bar{z}\in [\bar{y}]^\omega$ with $N(\bar{z}) = A$, so we have
		\[F'(\bar{y}) = F(N(\bar{y})) \neq F(A) = F'(\bar{z}),\]
		and so $\bar{y}$ is not homogeneous for $F'$, a contradiction.
	\end{proof}
    \begin{cor}\label{omegaplusk:cor}
        For any $k \in \omega$, \[\Rr \rightarrow (\omega + k)^{\omega + k} \iff \omega \rightarrow (\omega)^\omega\] and \[\Rr \rightarrow (\omega + k)^{\omega + k} \iff \omega \rightarrow (\omega)^\omega\]
    \end{cor}
    \begin{proof}
        This follows immediately from Lemma \ref{omega:lemma} by observing that for any $k \in \omega$, $\Rr \rightarrow (\omega + k)^{\omega+k} \iff \Rr \rightarrow (k + \omega^*)^{k + \omega^*} \iff \Rr \rightarrow (\omega)^\omega$.
    \end{proof}
	With Lemma \ref{omega:lemma} and Corollary \ref{omegaplusk:cor} we have shown one direction of Theorem~\ref{mainthm}, namely that if $\tau$ is of the form $\omega+k$ or $k+\omega^*$ for $k$ a natural number, then $\Rr \rightarrow(\tau)^\tau$ holds if and only if $\omega \rightarrow (\omega)^\omega$. To complete our proof of Theo\-rem~\ref{mainthm}, it therefore remains to show that $\Rr \centernot \rightarrow (\tau)^\tau$ for all other countably infinite order types $\tau$. We will show this by means of a much more general result which will also allow us to deduce Theorem \ref{solovay:theorem} as a corollary.
	\section{Negative relations}\label{negative:section}
	We begin this section with two classical results of Dushnik and Miller, phrased in terms of exactness (Definition \ref{exact:defn} in \S\ref{choice:subsection}):
	\begin{thm}[Dushnik \& Miller; {\cite[Theorem 1]{dushnikmiller}}]\label{dushnikmillercountable:thm} 
	    Every countably infinite linear order is inexact.
	\end{thm}
    \begin{thm}[Dushnik \& Miller; {\cite[Theorem 3]{dushnikmiller}}]\label{dushnikmilleruncountable:thm}
        \s{(ZFC)} There is an exact set of reals which has cardinality $2^{\aleph_0}$.
    \end{thm}
	Dushnik and Miller's construction of an exact set of reals makes essential use of a well-ordering of $\R$; we will show that it is consistent with \s{ZF} (relative to an inaccessible cardinal) that no infinite set of reals is exact. The following definition allows us to relate the notion of exactness to a well-studied regularity property for sets of reals:
    \begin{defn}\label{perfectset:defn}
        A set $P \subseteq \R$ is \emph{perfect} if it is closed and has no isolated points. A set $A \subseteq R$ has the \emph{Perfect Set Property (PSP)} if either $|A| \le \aleph_0$ or there is a non-empty perfect set $P \subseteq A$.
    \end{defn}
    \begin{prop}\label{pspcharacterisation:prop}
        \s{(ZF + AC$_{\omega}(\R)$)} Let $A \subseteq \R$ be uncountable. Then $A$ has the PSP if and only if $A$ has a subset which is order-isomorphic to $\Rr$.\footnote{To the best of our knowledge this equivalence in \s{ZFC} is a folklore result, but we could not easily find a proof in the literature. We include here a proof using only \s{AC$_\omega(\R)$}.}
    \end{prop}
    \begin{proof}
        $(\!\!\implies\!\!)\,$: Without loss of generality, let $A \subseteq \R$ be perfect. We use a ``Cantor scheme"-style construction to build a subset of $A$ which is order-isomorphic to $\twoomegalex$; this, in turn, will have a subset order-isomorphic to $\Rr$.

        Fix an enumeration $\langle O_n : n \in \omega\rangle$ of the open intervals with rational endpoints. Consider the family $\mathcal A \coloneqq \{O_n \cap A: n \in \omega, O_n \cap A \neq \emptyset\}$; using \s{AC$_\omega(\R)$}, we can find a choice function $\mathcal F : \mathcal A \to A$.

        We now recursively build two families, $\{I_s : s \in \lessomega\}$ and $\{x_s : s \in \lessomega\}$, with the following properties for each $s, t \in \lessomega$:
        \begin{enumerate}
            \item $I_s\subseteq \R$ is an open interval with rational endpoints such that $I_s \cap A \neq \emptyset$;
            \item $I_t \subseteq I_s$ if and only if $s \sqsubseteq t$;
            \item $x_s \in I_s \cap A$;
            \item $x_s < x_t$ if and only if $s \lex t$ in $\lessomega$, where here $s \lex t$ is taken to include the situations $t^\frown\langle0\rangle \sqsubseteq s$ and $s^\frown\langle 1 \rangle \sqsubseteq t$.
        \end{enumerate}
    Let $I_\emptyset$ be that $O_n$ with $n$ minimal such that $O_n \cap A \neq \emptyset$. Now, given $I_s$, we find $x_s, I_{s^\frown\langle0\rangle},$ and $I_{s^\frown\langle1\rangle}$ in the following way: let $n_s^0, n_s^1,$ and $n_s^2$ be such that the $O_{n_s^i}$ are pairwise disjoint subsets of $I_s$, each with diameter $< \frac{1}{3}\text{diam}(I_s)$, such that $O_{n_s^i} < O_{n_s^j}$ if $i < j$ and each $O_{n_s^i} \cap A \neq \emptyset$. We note that this is possible since none of the elements of $A$ are isolated points, and that no use of Choice is required here as we may take the $n_s^i$ to be minimal such in some fixed well-ordering of $\omega^3$. Now let $x_s = \mathcal F\left(O_{n^1_s} \cap A\right)$, $I_{s^\frown \langle 0 \rangle} = O_{n^0_s}$, and $I_{s^\frown \langle 1 \rangle} = O_{n^2_s}$. By construction, conditions 1--4 will be satisfied.

    Now, given $\{x_s : s \in \lessomega\}$, define
    \[A' \coloneqq \left\{\lim_{n\rightarrow\infty} x_{b \restriction n} : b \in \twoomega\right\}.\]
    Then by the construction of the $x_s$, $A'$ is order-isomorphic to $\twoomegalex$, and since $A$ is closed, $A' \subseteq A$.

    $(\:\!\!\!\!\impliedby\!\!\!\!\:)\,$: It suffices to show that any $A \in [\R]^\lambda$ has the PSP. We will show that for any such $A$ we can build a set $A' \subseteq \R$ such that (a) $A'$ is a perfect set; (b) $|A \triangle A'| \le \aleph_0$, and then observe that under such circumstances $A \cap A'$ must have the PSP.
    
    Let $A \in [\R]^\lambda$ and let $f : \R \to \R$ be an order-preserving function such that $A = f \im \R$. For $x \in \R$, write $f^-(x) \coloneqq \sup_{y < x} f(y)$ and $f^+(x) \coloneqq \inf_{y>x} f(y)$; then $f$ is continuous at $x$ if and only if $f^-(x) = f^+(x) = f(x)$. Note that $f$ has at most countably many discontinuities, since for any $x \in \R$ at which $f$ is discontinuous, there is an associated non-empty open interval $(f^-(x), f^+(x))$, which necessarily contains a rational. The value of $f(x)$ itself can be any of the values in this interval or either of its endpoints. Let $A'$ be the set defined from $A$ by replacing the image $f(x)$ of each discontinuity of $f$ by both of the lower and upper limits of $f$ at $x$:
    \[A' \coloneqq \{f^-(x) : x \in \R\} \cup \{f^+(x) : x \in \R\}.\]
    Certainly $|A \triangle A'| \le \aleph_0$, as for each of the 
    discontinuities of $f$, we have removed one element of $A$ and added two more. We claim that $A'$ is perfect. It is clear that no element of $A'$ is an isolated point. To see that $A'$ is closed, let $z$ be an element of the closure of $A'$; by \s{AC$_\omega(\R)$}, we may write $z = \lim_{n \rightarrow \infty} z_n$ as a limit of elements of $A'$. Then for each $n \in \omega$, either (i) $z_n = f(x_n)$ for some $x_n$ at which $f$ is continuous, or (ii) $z_n$ is a limit of such points, in which case we may find $x_n$ at which $f$ is continuous such that $|z_n - f(x_n)| < \frac{1}{n}$. Using \s{AC$_\omega(\R)$} to choose such an $x_n$ for each $z_n$ in case (ii), we may express $z$ in the form $z = \lim_{n \rightarrow \infty} f(x_n)$, where $f$ is continuous at each $x_n$. By construction, $A'$ contains all such limits, so $z \in A'$ and $A'$ is closed.

    Finally, we claim that the uncountable set $A \cap A'$ has the PSP, and so \textit{a fortiori} $A$ has the PSP also. This follows from the more general claim that if $P$ is a perfect set and $T \subseteq P$ is countable, then $P \setminus T$ has a perfect subset. To see this, reduce $P$ to some perfect $S = \{\lim_{n \rightarrow \infty} x_{b \restriction n} : b \in \twoomega\}$ with $\{x_s: s \in \lessomega\} \subseteq P$, as in the $(\!\!\implies\!\!)$ direction of this proof, and enumerate $T\cap S$ as $\{t_m : m \in \omega\}$. Then for each $m \in \omega$, let $b_m \in \twoomega$ be such that $t_m = \lim_{n\rightarrow\infty}x_{b_m\restriction n}$, and set
    \[P' \coloneqq \left\{\lim_{n\rightarrow\infty}x_{b \restriction n}: b \in \twoomega, b(2n) \neq b_{n}(2n)\right\}.\]
    Then $P' \subseteq S \setminus T \subseteq P \setminus T$ and $P'$ is perfect.
    \end{proof}
    \begin{cor}\label{pspsolovay:cor}
        \s{(ZF + AC$_\omega(\R)$)} If every set of reals has the PSP then there are no infinite exact sets of reals.
    \end{cor}
    \begin{proof}
         Let $A$ be an infinite set of reals. Since by assumption $A$ has the PSP, then either $A$ is countable, and so inexact by Theorem \ref{dushnikmillercountable:thm}, or it contains a perfect subset. By Proposition \ref{pspcharacterisation:prop}, it follows that if $A \subsetneq \R$ is uncountable, then we can find a subset of $A$ which is order-isomorphic to $\Rr$; since we can reduce this to a proper subset which is order-isomorphic to $A$, it follows that $A$ is inexact. Clearly $\R$ itself is also inexact.
    \end{proof}
    
    It follows from the above that the statement ``all infinite sets of reals are inexact" holds in two natural settings for the discussion of IEPRs. In Solovay's model, every set of reals has the PSP, and so every infinite set of reals is inexact. In particular, the statement that every infinite set of reals is inexact is consistent with \s{ZF} relative to an inaccessible cardinal. The same holds in models of \s{ZF + AD}: the Axiom of Determinacy \s{AD}, which states that every set of reals is determined, implies both that every set of reals has the perfect set property and that \s{AC$_\omega(\R)$} holds (see e.g.\ \cite[\S27]{kanamori}).
    We come now to our second major result, which underlies the proof of the remainder of Theorem \ref{mainthm} and the proof of Theorem \ref{solovay:theorem}. In contrast to Observation \ref{exacttrivial:obs}, this result shows that, essentially, if a real type $\tau$ has two inexact parts, then we can ``play them against each other" to define a colouring witnessing $\Rr \centernot \rightarrow (\tau)^\tau$.
	\begin{prop}\label{sumofinexacts:prop}
		Let $\tau$ be a real type. If $\tau$ can be written as a sum $\tau = \sigma_0 + \sigma_1$ where each $\sigma_i$ is inexact, then $\Rr \centernot \rightarrow (\tau)^\tau$.
	\end{prop}
	\begin{proof}
		As before, let $\langle q_n : n \in \omega\rangle$ be a fixed enumeration of $\Q$ and $\langle O_n : n \in \omega\rangle$ a fixed enumeration of the open intervals of $\R$ with rational endpoints.
		
		\begin{claim}\label{removable:claim} If $\sigma$ is an inexact real type and $X \in [\R]^\sigma$, then there exist $Y \in [X]^\sigma$ and an open interval $O \subseteq \R$ such that $O \cap X \neq \emptyset$ but $O \cap Y = \emptyset$.\end{claim}
		
		\begin{claimproof} Let $f : X \to X$ be an order-preserving map other than the identity, and let $z \in X$ be such that $f(z) \neq z$; wlog $f(z) > z$. Since $f$ is order-preserving, we have then that $z < f(z) < f(f(z))$. Define a function $g: X \to X$ by, for $x \in X$,
		\[g(x) = \begin{cases*}
			x & if $x < z$;\\
			f(f(x)) & if $x \ge z.$
		\end{cases*}
		\]
		This is easily checked to be order-preserving. Write $Y \coloneqq g\im X$. By definition of $g$, the open interval $O \coloneqq (z,f(f(z)))$ has empty intersection with $Y$, but its intersection with $X$ contains $f(z)$, so is non-empty.\end{claimproof}
		
		Let us say that such an $O$ \emph{meets $X$ removably}. Note that if $O$ meets $X$ removably, there is some $O_n \subseteq O$ an open interval with rational endpoints which also meets $X$ removably.
		
		Now let $\tau$ be a real type which can be decomposed as a concatenation of two inexact order types. Observe that for such a $\tau$, given any $A \in [\R]^\tau$, there exists some $x \in \R$ such that $A \cap (-\infty,x)$ and $A \cap (x,\infty)$ are both inexact. Say that such an $x$ \emph{separates} $A$. Now, for $A \in [\R]^\tau$, we define $x_A \in \R$ like so: if there is a unique $x \in \R$ such that $x$ separates $A$, then $x_A$ is that unique $x$; otherwise, the set of real numbers separating $A$ is convex, and we let $x_A$ be that $q_n$ with $n$ minimal such that $q_n$ separates $A$. Write
		\[\begin{split}
			A^- &\coloneqq A \cap (-\infty,x_A);\\
			A^+ &\coloneqq A \cap (x_A,\infty).
		\end{split}\]
		
		Now, for $X \subseteq \R$ with $X$ inexact, write
		\[m(X) \coloneqq \min\{n \in \omega: O_n \text{ meets }X\text{ removably}\},\]
		and define the colouring $F: [\R]^\tau\to 2$ by, for $A \in [\R]^\tau$,
		\[F(A) = \begin{cases*}
			0 & if $m(A^-) > m(A^+)$;\\
			1 & if $m(A^-) \le m(A^+)$.
		\end{cases*}\]
		We claim that this colouring has no homogeneous set. Suppose $A \in [\R]^\tau$; we will show that $A$ is not homogeneous. If all elements of $[A]^\tau$ have the property that only one real number separates them (case 1), then we work directly with this $A$; otherwise, some element of $[A]^\tau$ is separated by a rational (case 2): in this case we let $n^*$ be the minimal index of a rational separating any element of $[A]^\tau$, and assume, by reducing $A$ if necessary, that $q_{n^*}$ separates $A$, so $x_A = q_{n^*}$.\footnote{One can in fact show that $q_{n^*}$ necessarily separates $A$, but this makes the proof longer than it needs to be so is omitted.}
		
		Now, in either of the above cases, observe that we can find $B \in [A]^\tau$ with the property that $m(B^-) = m(A^-)$ and $m(B^+)$ is arbitrarily large in the following way: we restrict our attention to those $B$ with $B \cap (-\infty,x_A] = A \cap (-\infty,x_A]$ and $B \cap (x_A,\infty) \in [A^+]^{\text{otp}(A^+)}$, as any such $B$ is separated by $x_A$, and so in fact $x_B = x_A$ (by uniqueness in case 1, and by minimality in case 2). For such a $B$, $B^- = A^-$ and $B^+$ can range over any element of $[A^+]^{\text{otp}(A^+)}$. In particular, for any finite number of $O_n$ which meet $A^+$ removably, we can take $B^+$ to be an element of $[A^+]^{\text{otp}(A^+)}$ meeting none of them, and thereby make $m(B^+)$ as large as we like while keeping $m(B^-) = m(A^-)$; similarly we can find $B \in [A]^\tau$ with $m(B^-)$ arbitrarily large and $m(B^+) = m(A^+)$. In this way we can always move from any $A \in [\R]^\tau$ to some $B \in [A]^\tau$ with $F(B) \neq F(A)$.
	\end{proof}
	\begin{cor}\label{thm1negative:cor}
		If $\tau$ is a countably infinite order type not of the form $\omega + k$ or $k + \omega^*$ for some natural number $k$, then $\Rr \centernot \rightarrow (\tau)^\tau$.
	\end{cor}
	\begin{proof}
		Let $\tau$ be a countably infinite order type which cannot be written as a sum of two inexact order types; we will show that it is of the form $\omega + k$ or $k + \omega^*$ for some $k \in \omega$. Let $A \in [\R]^\tau$ and observe that there cannot be an $x \in \R$ separating $A$. Since every countably infinite linear order is inexact, it follows that for any $x \in \R$ it must be the case that precisely one of $(-\infty, x)\cap A$ or $(x,\infty)\cap A$ is infinite.
		
		If $\sup \{x \in \R: (x,\infty)\cap A \text{ is infinite}\} < \inf \{x \in \R:(-\infty,x) \cap A \text{ is infinite}\}$, then any real between these separates $A$, so they must be equal; say \begin{align*}
		    x^* \coloneqq &\sup \{x \in \R: (x,\infty)\cap A \text{ is infinite}\}\\ =& \inf \{x \in \R:(-\infty,x) \cap A \text{ is infinite}\}.
		      \end{align*}
        If $(-\infty,x^*)\cap A$ is infinite, then $(x^*,\infty)\cap A$ is finite and for any $z < x^*$, $(-\infty,z) \cap A$ is finite; it follows that $A$ has order type $\omega + k$ for some finite $k$. Similarly, if $(x^*,\infty)\cap A$ is infinite we get that $A$ has order type $k + \omega^*$ for some finite $k$.
	\end{proof}
	This completes the proof of Theorem \ref{mainthm}. Theorem \ref{solovay:theorem} will also be deduced as a corollary of Proposition \ref{sumofinexacts:prop} in much the same way.
	\begin{cor}\label{thm2:cor}
		If every infinite real type is inexact, then every uncountable real type $\tau$ is a sum of two inexact order types; in particular, $\Rr \centernot \rightarrow (\tau)^\tau$.
	\end{cor}
	\begin{proof}
		The proof is exactly as above for Corollary \ref{thm1negative:cor} using the extra assumption that \textit{all} infinite sets of reals are inexact, not just countable ones.
	\end{proof}
        \begin{proof}[Proof of Theorem \ref{solovay:theorem}]
            By Corollary \ref{thm2:cor}, if all sets of reals have the PSP, then for all uncountable $\tau$, $\Rr \centernot \rightarrow (\tau)^\tau$. For the converse, note that if there is an exact set of reals, then, writing $\tau$ for its order type, $\Rr \rightarrow (\tau)^\tau$ by Observation \ref{exacttrivial:obs}. Finally, the fact that every infinite set of reals is inexact if every set of reals has the PSP is precisely Corollary \ref{pspsolovay:cor}.
        \end{proof}
    
    By the discussion after Corollary \ref{pspsolovay:cor}, we see that in Solovay's model and in models of \s{ZF + AD}, $\Rr \centernot \rightarrow (\tau)^\tau$ for all uncountable $\tau$.

	\section{Larger homogeneous sets}\label{largerhomogeneous:section}
	So far in this paper we have focused on minimal relations, i.e.\ those of the form $\Rr \rightarrow(\tau)^\tau$. In this section, we present a full classification of the pairs of order types $\sigma,\tau$ with $\tau$ countably infinite for which $\Rr \rightarrow (\sigma)^\tau$ is consistent with \s{ZF + AC$_\omega(\R)$}, and we show that all of these relations hold in Solovay's model. It remains open whether the relations described in Proposition \ref{characterisationofctbleexp:prop}(c) are equivalent to $\omega \rightarrow (\omega)^\omega$.
    
    The results of this section can be summarised as follows:
    \begin{prop}\label{characterisationofctbleexp:prop}
        \s{(}\s{ZF + AC}$_\omega(\R)$\s{)} The only partition relations $\Rr \rightarrow (\sigma)^\tau$ with $\tau$ countably infinite which can hold are the following:
        \begin{enumerate}[(a)]
            \item $\Rr \rightarrow (\omega + n)^{\omega + k}$ and $\Rr \rightarrow (n + \omega^*)^{k + \omega^*}$ for $k \le n$ natural numbers;
            \item $\Rr \rightarrow (\omega + \alpha^*)^\omega$ and $\Rr \rightarrow (\alpha + \omega^*)^{\omega^*}$ for $\alpha$ a countable ordinal;
            \item $\Rr \rightarrow (\omega + \psi)^{\omega + 1}$ and $\Rr \rightarrow (\psi^* + \omega^*)^{1 + \omega^*}$ for $\psi$ any countable order type with $\omega + 1 \centernot \le \psi$.\footnote{The structure of such $\psi$ (reverse \emph{surordinals}) is discussed e.g.\ in \cite[pp. 814--815]{equimorphychains}, citing work of Jullien in \cite{jullienthesis}.}
        \end{enumerate}
        Moreover, (a) and (b) are equivalent to $\omega \rightarrow (\omega)^\omega$, and all three of (a), (b) and (c) hold in Solovay's model.
    \end{prop}
    We remark that in conjunction with Corollary \ref{pspsolovay:cor} and Corollary \ref{thm2:cor}, this gives a full characterisation of the relation $\Rr \rightarrow (\sigma)^\tau$ in Solovay's model. A brief note on the use of \s{AC}$_\omega(\R)$ here:
    \begin{obs}\label{ordinaliff:obs}
        \s{(}\s{ZF + AC}$_\omega(\R)$\s{)} For $\varphi$ a real type, $\varphi$ is an ordinal (in particular, a countable ordinal) iff $\omega^* \not \le \varphi$.\footnote{Without the assumption that $\varphi$ is a real type, this equivalence requires \s{DC}.}
    \end{obs}
    \begin{proof}
        Clearly if $\omega^* \le \varphi$ then $\varphi$ is not an ordinal, so focus on the other direction. Assume $\varphi$ is not well-ordered. Let $X \in [(0,1)]^\varphi$, so $\inf X$ exists. We will find a copy of $\omega^*$ in $X$. For $n \in \omega$ let $X_n \coloneqq \left(\inf X, \inf X + \frac{1}{n}\right)\cap X$. Then $\{X_n : n \in \omega\}$ is a countable family of non-empty sets of reals. By \s{AC}$_\omega(\R)$ we can choose an $x_n \in X_n$ for each $n \in \omega$; then $\{x_n : n \in \omega\} \subseteq X$ contains a copy of $\omega^*$.
    \end{proof}
    \begin{obs}\label{surordinalcountable:obs}
        \s{(}\s{ZF + AC}$_\omega(\R)$\s{)} For $\varphi$ a real type, if $1 + \omega^* \not \le \varphi$ then $\varphi$ is at most countable.
    \end{obs}
    \begin{proof}
        If $\omega^* \not \le \varphi$ then $\varphi$ is a countable ordinal by Observation \ref{ordinaliff:obs}, so assume $\omega^* \le \varphi$. Fix some $X \in [\R]^\varphi$ and fix a copy of $\omega^*$ in $X$, $x_0 > x_1 > x_2 > \dots,$ say. For $n \in \omega$, write $X_n \coloneqq X \cap (x_n,\infty)$. Then since $1 + \omega^* \not \le \varphi$, $X = \bigcup_{n \in \omega} X_n$, and none of the $X_n$ contain copies of $\omega^*$, so are isomorphic to countable ordinals by Observation \ref{ordinaliff:obs}. Then $X$ is a countable union of countable sets of reals; \s{AC}$_\omega(\R)$ implies that any such union is countable.
    \end{proof}
    We prove Proposition \ref{characterisationofctbleexp:prop} by means of a series of lemmas, all of which are results in \s{ZF} with no additional assumptions about Choice or Dedekind sets.
    \begin{lemma}\label{charpositive:lemma}
        For $k \le n$ natural numbers, \[\Rr \rightarrow (\omega + n)^{\omega + k} \iff \omega \rightarrow (\omega)^\omega\]
        and for any real type $\varphi$ with $\omega \centernot \le \varphi$,
        \[\Rr \rightarrow (\omega + \varphi)^\omega \iff \omega \rightarrow (\omega)^\omega.\]
    \end{lemma}
    \begin{proof}
        First observe that for $k \le n$ natural numbers, $\Rr \rightarrow (\omega + n)^{\omega + k}$ implies the minimal relation $\Rr \rightarrow (\omega + k)^{\omega+k}$, which by Corollary \ref{omegaplusk:cor} is equivalent to $\omega \rightarrow (\omega)^\omega$, so it suffices to show that $\omega \rightarrow (\omega)^\omega$ implies $\Rr \rightarrow (\omega + n)^{\omega + k}$. Appealing to Ramsey's theorem, one can show that $\omega \rightarrow (\omega)^\omega \implies \omega + \omega \rightarrow (\omega + n)^{\omega + k}$ (see \cite[Lemma 7.6]{kleinberg}); picking any copy of $\omega + \omega$ in $\R$, we deduce that \[\omega \rightarrow (\omega)^\omega \implies \Rr \rightarrow (\omega + n)^{\omega + k}.\]
        
        Now let $\varphi$ be a real type with $\omega \centernot \le \varphi$. Again, $\Rr \rightarrow (\omega + \varphi)^\omega$ implies the minimal relation $\Rr \rightarrow (\omega)^\omega$, which is equivalent to $\omega \rightarrow (\omega)^\omega$ by Lemma \ref{omega:lemma}. We therefore need only show that $\omega \rightarrow (\omega)^\omega \implies \Rr \rightarrow (\omega + \varphi)^\omega$.
        
        So suppose $\omega \rightarrow (\omega)^\omega$, so $\Rr \rightarrow (\omega)^\omega$. Then this relation also holds with $\Rr$ replaced by any isomorphic linear order, such as $\Rrminus$. Let $F : [\R]^\omega \to 2$ be a colouring. Then $F \restriction [\R_{<0}]^\omega$ has a homogeneous set $H \in [\R_{<0}]^\omega$, as $\Rrminus \rightarrow (\omega)^\omega$; letting $K$ be any copy of $\varphi$ in $\Rrplus$, we can ``pad" $H$ out to $H \cup K$ without breaking its homogeneity for $F$, as every element of $[H\cup K]^\omega$ is necessarily an element of $[H]^\omega$. Thus, there is a homogeneous set for $F$ of order type $\omega + \varphi$, and since $F$ was arbitrary we obtain $\Rr \rightarrow (\omega + \varphi)^\omega$.
    \end{proof}
    With the additional assumption of \s{AC}$_\omega(\R)$, a real type $\varphi$ has $\omega \centernot \le \varphi$ iff it is the reverse of a countable ordinal, by Observation \ref{ordinaliff:obs}. Appealing to the symmetry of $\Rr$, we obtain the equivalence of the relations in (a) and (b) of Proposition \ref{characterisationofctbleexp:prop} with $\omega \rightarrow (\omega)^\omega$. Next we show that nothing more is possible for exponents other than $\omega + 1$ and $1 + \omega^*$:
    \begin{lemma}\label{charnegative:lemma}
        For any $k \in \omega$,
        \begin{align*}
            \Rr &\centernot \rightarrow (\omega^* + \omega + k)^{\omega+k}\text{ and}\\
            \Rr &\centernot \rightarrow (\omega + \omega + k)^{\omega + k}.
        \end{align*}
        Moreover, if $k \ge 2$,
        \begin{align*}
            \Rr &\centernot \rightarrow (\omega + \omega)^{\omega + k}\text{ and}\\
            \Rr &\centernot \rightarrow (\omega + \omega^*)^{\omega + k}.
        \end{align*}
    \end{lemma}
    \begin{proof}
    Let $k \in \omega$ and for any $\bar{x} \in [\R]^{\omega + k}$ write $\bar{x}= \{x_\alpha : \alpha < \omega + k\}$. The colouring $F : [\R]^{\omega + k} \to 2$ given by, for $\bar{x} \in [\R]^{\omega+k}$,
	\[F(\bar{x}) = \begin{cases*}
		0 & if $N(x_0,x_1) > N(x_2,x_3)$;\\
		1 & if $N(x_0,x_1) < N(x_2,x_3)$
	\end{cases*}\]
	has no homogeneous set of order type $\omega^* + \omega + k$ or $\omega + \omega + k$, so we obtain $\Rr \centernot \rightarrow (\omega^* + \omega + k)^{\omega + k}$ and $\Rr \centernot \rightarrow (\omega + \omega + k)^{\omega + k}$.

    Now, if moreover $k \ge 2$, the colouring $G : [\R]^{\omega + k} \to 2$ given by, for $\bar{x} \in [\R]^{\omega+k}$,
	\[G(\bar{x}) = \begin{cases*}
		0 & if $N(x_0,x_1) > N(x_{\omega},x_{\omega+1})$;\\
		1 & if $N(x_0,x_1) < N(x_{\omega},x_{\omega+1})$
	\end{cases*}\]
    has no homogeneous set of order type $\omega + \omega$ or $\omega + \omega^*$, and so $\Rr \centernot \rightarrow (\omega + \omega)^{\omega+k}$ and $\Rr \centernot \rightarrow (\omega + \omega^*)^{\omega + k}$.
    \end{proof}
	Symmetric results hold for $k + \omega^*$, and so we have completed the proof of Proposition \ref{characterisationofctbleexp:prop} (a) and (b). We remark here that for for any $\varphi$ which is the order type of a Dedekind set of reals, the colouring $F$ has no homogeneous set ordered as $\varphi + \omega + k$, and the colouring $G$ has no homogeneous set ordered as $\omega + \varphi$, and so in fact, for $1 \neq k \in \omega$, (a) and (b) characterise those $\sigma$ for which the relation $\Rr \rightarrow (\sigma)^{\omega + k}$ is consistent in \s{ZF} (although (b) needs to be rephrased to ``$\Rr \rightarrow (\omega + \varphi)^\omega$ for all $\varphi$ with $\omega \centernot \le \varphi$").

    We turn now to part (c) of Proposition \ref{characterisationofctbleexp:prop}. For some time the question of whether $\Rr \rightarrow (\omega + \sigma)^{\omega +1}$ was possible for some infinite $\sigma$ was open. 
	Recently, this question has been answered in the positive: the author obtained a re-characterisation of the question in terms of colourings of $[\omega]^\omega$ and Schilhan pointed out that this equivalent statement holds in Solovay's model. 
    The argument is most neatly phrased in terms of \textit{polarised partition relations}:
    \begin{defn}
        For $\Lzero, \Lone$ linear orders, and $\sigma_0,\sigma_1,\tau_0, \tau_1$ order types, the \emph{polarised partition relation}
        \[\binom{\Lzero}{\Lone} \rightarrow \binom{\sigma_0}{\sigma_1}^{\tau_0,\tau_1}\]
        is the statement that for any colouring $F : [L_0]^{\tau_0} \times [L_1]^{\tau_1} \to 2$, there are some $H_0 \in [L_0]^{\sigma_0}$ and $H_1 \in [L_1]^{\sigma_1}$ with the property that $\left| F \im ([H_0]^{\tau_0} \times [H_1]^{\tau_1})\right| = 1$. The pair $H_0$, $H_1$ is said to be \emph{homogeneous} for $F$.
    \end{defn}
    \begin{lemma}\label{recharonreals:lemma}
    Let $\psi$ be a real type with $\omega + 1 \centernot \le \psi$. Then
        \[\Rr \rightarrow (\omega + \psi)^{\omega + 1} \iff \binom{\Rr}{\langle\R_{\ge 0},<\rangle} \rightarrow \binom{\omega}{\psi}^{\omega,1}.\]
    \end{lemma}
    \begin{proof}
	$(\!\!\implies\!\!)\,$: Assume $\Rr \rightarrow (\omega + \psi)^{\omega + 1}$ and let \[F : [\R]^\omega \times [\R_{\ge 0}]^1 \to 2\] be a colouring. We wish to find some $H_0 \in [\R]^\omega$, $H_1 \in [\R_{\ge 0}]^\psi$ such that $F$ is constant on $[H_0]^\omega\times [H_1]^1$. $F$ induces a colouring $F' : [\R]^{\omega+1} \to 2$ in the following way: for $\bar{x} = \langle x_0,x_1,x_2,\dots,x_\omega\rangle,$ we treat the maximal element $x_\omega$ as a ``shifted" element of $\R_{\ge 0}$; write $x' \coloneqq \sup_{n < \omega}x_n$, so $x_\omega - x' \in [0,\infty)$, and set
	\[F'(\bar{x})\coloneqq F(\bar{x}\restriction \omega,\langle x_\omega - x'\rangle).\]
	By assumption, this colouring has a homogeneous set $H \in [\R]^{\omega + \psi}$. Write $H = H_0 \cup K$, where $H_0$ is the initial segment of $H$ ordered as $\omega$ and $K$ is the final segment of $H$ above $H_0$, ordered as $\psi$. Now, writing $h \coloneqq \sup H_0$, set $H_1 \coloneqq \{x \in \R_{\ge 0} : x + h \in K\}$. Then $F$ is constant on $[H_0]^\omega\times [H_1]^1$, as required.
		
	$(\!\!\impliedby\!\!)\,$: Let $G: [\R]^{\omega+1} \to 2$. We use this to induce a colouring $G': [\R_{<0}]^\omega \times [\R_{\ge 0}]^1 \to 2$ in the following way: for $\bar{x} \in [\R_{<0}]^\omega$, $y \in \R_{\ge 0}$,
 
    \[G'(\bar{x},\langle y\rangle) \coloneqq G(\bar{x}^\frown \langle y \rangle).\]
    Since $\Rrminus$ is order-isomorphic to $\Rr$, our assumption that
    
    \[\binom{\Rr}{\langle\R_{\ge 0},<\rangle} \rightarrow \binom{\omega}{\psi}^{\omega,1}\]
    gives some $H_0 \in [\R_{<0}]^\omega$, $H_1 \in [\R_{\ge 0}]^\psi$ homogeneous for $G'$, i.e.\ with the property that $\left|G' \im ([H_0]^\omega\times [H_1]^1)\right| = 1$; but since $H_0 \subseteq \R_{<0}$ and $H_1 \subseteq \R_{\ge 0}$, $H \coloneqq H_0 \cup H_1$ has order type $\omega + \psi$; since any element of $[H]^{\omega+1}$ consists of an $\omega$-sequence from $H_0$ followed by a single element of $H_1$, it follows by definition of $G'$ that $\left| G \im [H]^{\omega + 1} \right| = 1$.
    \end{proof}
    \begin{lemma}\label{recharonomega:lemma} For any real type $\psi$,
	\[\binom{\Rr}{\langle\R_{\ge 0},<\rangle} \rightarrow \binom{\omega}{\psi}^{\omega,1} \iff \binom{\omega}{\langle\R_{\ge 0},<\rangle} \rightarrow \binom{\omega}{\psi}^{\omega,1}.\]
    \end{lemma}
    \begin{proof}
	This is similar to the proof of Lemma \ref{omega:lemma}.
	
	$(\!\!\impliedby\!\!)\,$: A colouring $F: [\R]^\omega \times [\R_{\ge 0}]^1 \to 2$ restricts to a colouring $F' : [\omega]^\omega \times [\R_{\ge 0}]^1 \to 2$. By assumption there are some $H_0 \in [\omega]^\omega$ and $H_1 \in [\R_{\ge 0}]^\psi$ homogeneous for $F'$, i.e.\ with $\left|F' \im [H_0]^\omega \times [H_1]^1\right| = 1$; but then $H_0, H_1$ are homogeneous for $F$ also.
		
	$(\!\implies\!)\,$: We prove the contrapositive, following the proof of Lemma \ref{omega:lemma}. As in that proof, for any colouring $G : [\omega]^\omega \times [\R_{\ge 0}]^1 \to 2$ we get an induced colouring $G' : [\R]^\omega \times [\R_{\ge 0}]^1 \to 2$ by, in the first coordinate, colouring the $N$-sequences of $N$-increasing $\omega$-sequences of reals.
    
    Now, if $G : [\omega]^\omega \times [\R_{\ge 0}]^1 \to 2$ is a colouring witnessing that
    \[\binom{\omega}{\langle\R_{\ge 0},<\rangle} \centernot \rightarrow \binom{\omega}{\psi}^{\omega,1},\]i.e.\ is such that there are no $H_0 \in [\omega]^\omega$, $H_1 \in [\R_{\ge 0}]^\psi$ homogeneous for $G$, then the corresponding induced colouring $G'$ will also have no homogeneous pair $H_0 \in [\R]^\omega$, $H_1 \in [\R_{\ge 0}]^\psi$, and so witnesses that
    \[\binom{\Rr}{\langle\R_{\ge 0},<\rangle} \centernot\rightarrow \binom{\omega}{\psi}^{\omega,1}.\]
    \end{proof}
    \begin{prop}\label{schilhan:prop}
	In Solovay's model,
    \[\binom{\omega}{\langle\R_{\ge 0},<\rangle} \rightarrow \binom{\omega}{\eta}^{\omega,1},\]
    where $\eta$ is the order type of $\langle \Q,<\rangle$. In particular, in Solovay's model it holds that for any countable $\psi$ with $\omega + 1 \centernot \le \psi$,  $\Rr \rightarrow (\omega + \psi)^{\omega+1}$.
    \end{prop}
    \begin{proof}
	This argument is due to Jonathan Schilhan and we include it with his permission. It is based on the argument from \cite{mathias} that $\omega \rightarrow (\omega)^\omega$ in Solovay's model. Work in the model $W \coloneqq L(\R)^{V[G]}$, where $V$ is a ground model containing an inaccessible cardinal $\kappa$ and $G$ is a generic filter over $V$ for the Lévy collapse forcing Coll$(\omega,<\!\!\kappa)$. In $W$, let $F : [\omega]^\omega \times [\R_{\ge 0}]^1 \to 2$ be a colouring, 
    and let $a \in \R$, $\Phi$ a formula be such that for all $x \in [\omega]^\omega, z \in \R$,
	\[W \models F(x,\langle z \rangle) = 1 \iff V[a][x][z] \models \Phi(a,x,z).\]
    In particular, we have for $z \in \R^{V[a]}$ that
        \[W \models F(x, \langle z \rangle) = 1 \iff V[a][x]\models \Phi(a,x,z).\]
    We will think of the second component as indexing a collection of $\continuum$-many colourings of $[\omega]^\omega$. For $z \in \R_{\ge 0}$, write $F_z$ for the colouring $F_z : [\omega]^\omega \to 2$ given by
    \[F_z(x) \coloneqq F(x,\langle z \rangle).\]
	
    Then we can rephrase the problem like so: we wish to find some $H_1 \in [\R_{\ge 0}]^\eta$ such that there is some $H_0 \in [\omega]^\omega$ which is simultaneously homogeneous for each $F_z$ with $z \in H_1$ (with the same colour for each $z$). We proceed to show that we can find exactly this, with $H_1$ in fact consisting entirely of members of $V[a]$. Let $\mathbb M$ denote Mathias forcing in $V[a]$.
            
    \begin{claim}\label{mathiasreal:claim} If $r \in [\omega]^\omega \cap W$ is a Mathias real over $V[a]$ then for all $z \in \R_{\ge 0}^{V[a]}$, $r$ is almost homogeneous for $F_z$, i.e.\ for each $z \in \R_{\ge 0}^{V[a]}$, there is some $n \in \omega$ such that $r \setminus n$ is homogeneous for $F_z$.\end{claim}
    \begin{claimproof}
    We prove the following slightly more general claim: let $\Psi$ be any formula in two free variables, potentially with parameters in $V[a]$ which we suppress. Then a Mathias real $r$ over $V[a]$ has the property that for some $n \in [\omega]$, every $x \in [r\setminus n]^\omega$ agrees on the question of whether $V[a][x] \models \Psi(a,x)$. To complete the proof we consider the particular case where $\Psi(a,x)$ is $\Phi(a,x,z)$ as defined above.

    For any $A \in [\omega]^\omega \cap V[a]$, consider the $\mathbb M$-condition $\langle \emptyset, A\rangle$. By the fact that Mathias forcing $\mathbb M$ has pure decision, there is an $\mathbb M$-condition extending $\langle \emptyset, A\rangle$ with the same stem, i.e.\ a condition $\langle \emptyset,B\rangle$ for some $B \in [A]^\omega \cap V[a]$, such that $\langle \emptyset,B\rangle \parallel \Psi(\check{a},\dot{r})$, where $\dot{r}$ is the canonical name for the Mathias real. It follows that $\mathcal D \coloneqq \{\langle s,A\rangle : \langle \emptyset, A\rangle \parallel \Psi(\check{a},\dot{r})\}$ is dense, so meets any generic filter. But now if $r$ is a Mathias real over $V[a]$ with corresponding generic filter $H \subseteq \mathbb{M}$ and $\langle s,A\rangle \in H \cap \mathcal D$, then $r' \coloneqq r \setminus s = r \setminus (\max s + 1)$ is a Mathias real over $V[a]$ with $\langle \emptyset, A\rangle$ in its generic filter; since $\langle \emptyset,A \rangle\parallel \Psi(\check{a},\dot{r})$ and every infinite subset of $r'$ is again a Mathias real with $\langle \emptyset, A \rangle$ in its generic filter, we have that every $x \in [r']^\omega$ agrees on whether $V[a][x]\models \Psi(a,x)$.\end{claimproof}

	Since the set of dense subsets of $\mathbb M$ in $V[a]$ is countable in $W$, $W$ contains Mathias reals over $V[a]$. Fix some $r \in [\omega]^\omega \cap W$ a Mathias real over $V[a]$ and work now in $V[a][r]$. For any $z \in \R^{V[a]}$, let $n_z \in \omega$ be minimal such that $r\setminus n_z$ is homogeneous for the colouring $F_z$. Now, in $V[a][r]$, $\R^{V[a]}$ is still uncountable, as Mathias forcing is proper. By the pigeonhole principle in $V[a][r]$, there is therefore some $n^* \in \omega$ and a set $A \subseteq \R^{V[a]}$ which is uncountable in $V[a][r]$ with the property that for all $z \in A$, $n_z = n^*$. In particular, $A$ contains a subset ordered as $\eta$, $A'$, say. Then in $W$, $A' \in [\R_{\ge 0}]^\eta$ has the property that $r \setminus n^*$ is homogeneous for every $F_z$ with $z \in A'$; the colour taken by $F_z$ on $[r \setminus n^*]^\omega$ is not necessarily the same for all $z \in A'$, but we can reduce again to some $H_1 \in [A']^\eta$ such that this colour is the same for all $z \in H_1$. Now, setting $H_0 \coloneqq r \setminus n^*$, we have found a homogeneous pair for $F$ of the desired order types.

    In particular, for $\psi$ any countable order type with $\omega + 1 \not \le \psi$ we have $\psi \le \eta$ (by the universality of $\eta$) and so, applying Lemma \ref{recharonreals:lemma} and Lemma \ref{recharonomega:lemma}, we conclude that $\Rr \rightarrow (\omega + \psi)^{\omega + 1}$.
    \end{proof}

    We remark that we could have started with the assumption that $\psi$ was \textit{any} real type with $\omega + 1 \not \le \psi$, as any such $\psi$ is countable under \s{AC}$_\omega(\R)$ by Observation \ref{surordinalcountable:obs}.

    William Chan has pointed out that the proof of Proposition \ref{schilhan:prop} can be adapted to the setting of \s{AD}$^+$ or \s{AD}$_\R$: using the existence of an $\infty$-Borel code $(S,\Phi)$ for the colouring, one can run the same argument, replacing the model $V[a]$ with the model $L[S]$. This answers a question from the original version of this paper, and characterises the IEPRs which hold on $\Rr$ under \s{AD}$^+$ or \s{AD}$_\R$: they are identical to those which hold in Solovay's model (i.e.\ the relations with $\tau$ countable mentioned in Proposition \ref{characterisationofctbleexp:prop}, and no relations with $\tau$ uncountable), as each of $\omega \rightarrow (\omega)^\omega$, \s{AC}$_\omega(\R)$, and the statement that all sets of reals are inexact follow from either \s{AD}$^+$ or \s{AD}$_\R$.
    
    \section{Further work}\label{furtherwork:section}
    A natural direction for further exploration is to consider infinite-exponent partition relations on other ill-founded orders. We have obtained a number of results concerning such IEPRs on generalisations of the reals, by which we mean orders of the form $\twoalphalex$ for $\alpha > \omega$ an ordinal, by combining ideas from this paper with ideas from \cite{crtolo}; these results will appear in \cite{higheranalogues}. Some questions still remain even in the setting of countable-exponent relations on the reals:
    \begin{q}
        Which of the following statements are equivalent?
        \begin{equation}
            \binom{\omega}{\langle\R_{\ge 0},<\rangle} \rightarrow \binom{\omega}{\eta}^{\omega,1}
        \end{equation}
        \begin{equation}
            \Rr \rightarrow (\omega + \omega)^{\omega + 1}
        \end{equation}
        \begin{equation}
            \Rr \rightarrow (\omega + \omega \lor \omega + \omega^*)^{\omega + 1}
        \end{equation}
        \begin{equation}
            \omega \rightarrow (\omega)^\omega
        \end{equation}
    \end{q}
    An equivalent formulation is ``Which of the implications $(1) \implies (2) \implies (3) \implies (4)$ can be reversed?". Here the $\lor$ in relation (3) indicates that the homogeneous set can have \textit{either} order type $\omega + \omega$ or $\omega + \omega^*$.
    \begin{q}
        Is the assumption of \s{AC}$_\omega(\R)$ necessary for Proposition \ref{characterisationofctbleexp:prop}(c)? Is it consistent with \s{ZF} that for $\psi$ the order type of some Dedekind set of reals,
        \[\Rr \rightarrow (\omega + \psi)^{\omega + 1}?\]
    \end{q}
    %
    In the setting of uncountable exponents we have the following question, which would generalise Theorem \ref{mainthm} and provide a full classification of the relation $\Rr \rightarrow (\tau)^\tau$:
    \begin{q}
	If $\tau$ is an inexact real type which cannot be written as a sum of two inexact order types, is it the case that \[\Rr \rightarrow (\tau)^\tau \iff \omega \rightarrow (\omega)^\omega?\]
    \end{q}

    Finally, the following question which is unrelated to partition relations naturally presents itself in light of Corollary \ref{pspsolovay:cor}:
    \begin{q}
        What is the consistency strength of the statement ``every infinite set of reals is inexact"? Can this statement be separated from the statement ``every set of reals has the PSP"?
    \end{q}
    On the level of individual sets, inexactness is a strict consequence of the PSP; under \s{PFA}, an $\aleph_1$-dense set of reals is inexact but does not have the PSP, as it has cardinality strictly between $\aleph_0$ and $2^{\aleph_0} = \aleph_2$.
    
\end{document}